\documentclass{amsart}[12pt]
\usepackage{amsmath,amssymb,amsthm,amscd}

\newtheorem{theorem}{Theorem}

\newtheorem{proposition}[theorem]{Proposition}

\newtheorem*{remark}{Remark}
\newtheorem*{ack}{Acknowledgements}
\newcommand{\Z}{\mathbb{Z}}

\newcommand{\F}{\mathbb{F}}
\newcommand{\Q}{\mathbb{Q}}
\newcommand{\C}{\mathbb{C}}
\renewcommand{\P}{\mathbb{P}}

\newcommand{\Gal}{{\rm Gal}}

\newcommand{\PSL}{{\rm PSL}}
\newcommand{\GL}{{\rm GL}}

\newcommand{\legen}[2]{\genfrac{(}{)}{}{}{#1}{#2}}
\newcommand{\im}{{\rm im}~}
\newcommand{\tr}{{\rm tr}~}
\newcommand{\sign}{{\rm sign}}
\begin{document}

\title[Trinomials]{Trinomials defining quintic number fields}
\author{Jesse Patsolic}
\author{Jeremy Rouse}
\address{Department of Mathematics, Wake Forest University, Winston-Salem, NC 27109}
\email{jesse.l.patsolic@alumni.wfu.edu}
\email{rouseja@wfu.edu}

\begin{abstract}
Given a quintic number field $K/\Q$, we study the
set of irreducible trinomials, polynomials of the form $x^{5} + ax + b$, that
have a root in $K$. We show that there is a genus four curve $C_{K}$
whose rational points are in bijection with such trinomials. This curve
$C_{K}$ maps to an elliptic curve defined over a number field, and using
this map, we are able (in some cases) to determine all the rational points
on $C_{K}$ using elliptic curve Chabauty. 
\end{abstract}

\maketitle

\section{Introduction and Statement of Results}

Number fields definied by trinomials, polynomials of the shape, 
$f(x) = ax^{n} + bx + c$, have a long history. Two such trinomials
$f_{1}(x)$ and $f_{2}(x)$ are equivalent if $f_{2}(x) = \alpha f_{1}(\beta x)$
for some $\alpha, \beta \in \Q^{\times}$. If $n = 5$, the Galois group
of a trinomial is contained in $F_{20}$, the Frobenius group of order $20$,
if and only if it is equivalent to one of the form
\[
  (4u^{2} + 16) x^{5} + (5u^{2} - 5) x + (4u^{2} + 10u + 6),
\]
(see \cite{Weber} section 189, \cite{Matzat2}, pages 90-91, and \cite{SpearmanWilliams}). Moreover, the Galois group is contained
in the dihedral group of order $10$ if and only if $u = t - 1/t$.

Malle \cite{Malle} determined the family of degree $6$ trinomials with Galois 
group contained in $S_{5}$, namely $(125-u)x^{6} + 12u(u+3)^{2} x + u(u-5)(u+3)^{2}$. Many transitive subgroups of $S_{6}$ occur as Galois groups of degree
$6$ trinomials, including for example $\Z/6\Z$ which is the Galois group
of $f(x) = x^{6} + 133x + 209$. (This is the only sextic trinomial, up
to equivalence, with this Galois group, as shown in \cite{BremnerSpearman}.)

In 1969, Trinks discovered the trinomial $f(x) = x^{7} - 7x + 3$ which
was proven (by Matzat) to have Galois group
$\GL_{3}(\Z/2\Z) \cong \PSL_{2}(\Z/7\Z)$, the unique finite simple
group of order $168$.  In 1979, Erbach, Fischer, and McKay \cite{EFM}
found a second such trinomial, $f(x) = x^{7} - 154x + 99$. Finally, in
1999, Elkies and Bruin \cite{ElkiesBruin} proved that the set of such
trinomials with Galois group $\GL_{3}(\Z/2\Z)$ is in bijection with
the rational points on a genus $2$ curve, and there are exactly four
equivalence classes of such trinomials (including, of course, the two
mentioned above).  They also studied degree $8$ trinomials with Galois
group contained in $(\Z/2\Z)^{3} \rtimes \GL_{3}(\Z/2\Z)$ and found
four examples, including $x^{8} + 324x + 567$, whose Galois group is
$\GL_{3}(\Z/2\Z)$. The classification of finite simple groups puts a
lot of restrictions on the Galois group that an irreducible trinomial
$x^{n} + ax^{s} + b$ can have. In particular (see \cite{Feit}), if
$n > 11$ is prime then the Galois group $G$ is either solvable,
isomorphic to $A_{n}$ or $S_{n}$, or $n = 2^{e} + 1$ is a Fermat prime
and
$\GL_{2}(\F_{2^{e}}) \subseteq G \subseteq \F_{2^{e}}^{2} \rtimes
\GL_{2}(\F_{2^{e}})$.

In this paper, we instead study the following question. Fix a degree $5$
number field $K/\Q$. What are all the irreducible trinomials $x^{5} + ax + b$
that have a root in $K$? We make progress on answering this question.
The following result is a summary of what we accomplish.

\begin{theorem}
\label{main}
Let $K = \Q[\alpha]$ be a degree $5$ extension of $\Q$,
and let $g(x) \in \Q[x]$ be the minimal polynomial of $\alpha$. Then,
\begin{itemize}
\item There is a genus four curve $C_{K}/\Q$ whose rational points $C_{K}(\Q)$
are in bijection with equivalence classes of irreducible trinomials
$f(x) = x^{5} + ax + b$ so that $K \cong \Q[x]/(f(x))$. 
\item If $L$ is the smallest field over which $g(x)$ factors into a quadratic
and a cubic in $L[x]$, there is an elliptic curve $E/L$
and a degree $2$ map $\phi : C_{K} \to E$ defined over $L$. The curve
$E$ has $j$-invariant $-25/2$.
\item Assuming that the Galois closure of $K$ over $\Q$ has Galois group
$F_{20}$, $A_{5}$ or $S_{5}$ and $K/\Q$ is unramified at $2$ and $5$,
the root number of $E/L$ is even.
\item There is a degree $60$ map $\psi : E \to \P^{1}$ with the property
that if $P \in C_{K}(\Q)$, then $\psi \circ \phi(P) \in \P^{1}(\Q)$.
\end{itemize} 
\end{theorem}

\begin{remark}
  Since $C_{K}$ has genus four, Faltings's theorem \cite{Faltings}
  proves that $C_{K}(\Q)$ is finite, and so there are only finitely
  many irreducible trinomials (up to equivalence) with a root in
  $K$. Moreover, Caporaso, Harris and Mazur showed \cite{CHM} that the
  weak Lang conjecture (that the set of rational points on a variety
  of general type is not Zariski dense) implies a bound on the
number of rational points on a curve of genus $g$ depending only on $g$.
\end{remark}

The map $\psi : E \to \P^{1}$ above provides a convenient way to study
the rational points on $C_{K}$.  In particular, if it is possible to
find generators for $E(L)$, then the method of elliptic curve Chabauty
(see \cite{ECChab}) allows one to provably determine all of the
rational points on $C_{K}$.

\begin{theorem}
\label{examples}
\begin{enumerate}
\item Let $f(x) = x^{5} - 5x + 12$, and $\alpha$ be a root of $f(x)$.
The field $K = \Q[\alpha]$ is a number field whose Galois closure
has dihedral Galois group. Up to equivalence, the only trinomial with
a root in $K$ is $f(x)$.
\item Assume the generalized Riemann hypothesis. Let
  $K = \Q[\sqrt[5]{18}]$. Up to equivalence, the only trinomials
with a root in $K$ are the following:
\begin{align*}
  f_{1}(x) &= x^{5} - 18\\
  f_{2}(x) &= x^{5} - 324\\
  f_{3}(x) &= x^{5} - 24\\
  f_{4}(x) &= x^{5} - 432\\
  f_{5}(x) &= x^{5} + 750x + 3750.
\end{align*}
The fifth trinomial above is mentioned in Example 3.1 of \cite{SW2}, page 385.
\end{enumerate}
\end{theorem}

\begin{remark}
  The main difficulty in applying elliptic curve Chabauty is finding
  generators for $E(L)$, because point searching over number fields of
  degree $10$ is very difficult. The easiest way to arrange this is to
  find lots of points on $E$ by taking images of rational point on
  $C_{K}$.  A priori, elliptic curve Chabauty can be successful if
  ${\rm rank}~E(L) < [L : \Q]$. However, the map $C_{K}/\Q \to E/L$
  gives a map from ${\rm Jac}(C) \to {\rm Res}_{L/\Q}(E)$. If the rank
  of ${\rm Jac}(C)(\Q)$ is greater than $3$, then elliptic curve
  Chabauty will not succeed. (This was explained to the authors by
  Nils Bruin.)  In the case that $K = \Q[\alpha]$, where $\alpha$ is a
  root of $f(x) = x^{5} + 2x^{4} - 11x^{3} - 17x^{2} + 26x + 55$,
  there are five obvious rational points on $C_{K}$ and (assuming GRH)
  the rank of ${\rm Jac}(C)$ is equal to that of ${\rm Res}_{L/\Q}(E)$
  which is equal to $4$ (as confirmed by a long $2$-descent
  computation). In this case, elliptic curve Chabauty will not be
  successful.
\end{remark}

\begin{remark}
The largest size of $C_{K}(\Q)$ we have found is $8$. In particular,
if $K = \Q[\alpha]$, where $\alpha$ is a root of 
$f_{1}(x) = x^{5} + 75x + 105$, then $K$ also contains roots of
\begin{align*}
  f_{2}(x) &= x^{5} - 75x + 465\\
  f_{3}(x) &= x^{5} - 1125x + 3825\\
  f_{4}(x) &= x^{5} - 2025x + 65205\\
  f_{5}(x) &= x^{5} + 2025x + 10665\\
  f_{6}(x) &= x^{5} - 10125x + 83025\\
  f_{7}(x) &= x^{5} + 28125x - 39375\\
  f_{8}(x) &= x^{5} - 3410625x + 86685375.
\end{align*}
\end{remark}

The example above raises the following question: what is the largest
number of rational points $C_{K}$ can have? Are there families of
quintic fields that each have many trinomials?

In \cite{SW2}, Spearman and Williams gave the example
that if $r \ne 0, \pm 1$ is rational, and $\alpha^{5} = r^{3} (r+1)(r-1)^{4}$,
then $\Q[\alpha]$ also contains the roots of $x^{5} + ax + b$, where
\begin{align*}
  a &= \frac{-80r(r^{2}-1)(r^{2}+r-1)(r^{2}-4r-1)}{(r^{2}+1)^{4}}\\
  b &= \frac{-32r(r^{2}-1)(r^{4}+22r^{3}-6r^{2}-22r+1)}{(r^{2}+1)^{4}}.
\end{align*}
We give a different one-parameter family of number fields that each
have (at least) two trinomials $x^{5} + ax + b$ in them with $a$ and
$b$ both nonzero.

\begin{theorem}
\label{family}
Let $a \in \Q$ with $a \ne -8$. Suppose that
\[
  f(x) = (4a + 32)x^{5} + (-5a^{2} + 5a)x - a^{3} + a^{2}
\]
is irreducible. Let $\alpha$ be a root of $f(x)$. Then the number
\[
  \frac{1}{a^{2}+4a-8} \left[ \frac{4a^{2} + 16a - 128}{a^{2} - a} \alpha^{4}
  + \frac{8a+64}{a} \alpha^{3} + (-2a-16) \alpha^{2} + (2a + 4) \alpha - 4a + 16 \right]
\]
is a root of
\[
  (a^{3} + 7a^{2} - 8a) x^{5} + (10a^{2} + 115a - 125)x + 2a^{2} - 76a - 250.
\]
\end{theorem}
(We omit the proof of this theorem, as it is a straightforward calculation.)

In Section 2, we recall background on trinomials, the canonical ring
of a curve, root numbers of elliptic curves, and elliptic curve
Chabauty. In Section 3, we prove parts 1 and 2 of
Theorem~\ref{main}. In Section 4, we prove parts 3 and 4 of
Theorem~\ref{main}. In Section 5, we give prove
Theorem~\ref{examples}. In Section 6, we discuss a K3 surface $X$
whose rational points correspond to number fields $K$ that have more
than one trinomial (up to equivalence) defining
them. Theorem~\ref{family} results from a rational curve on this
surface.

\begin{ack}
  We used Magma \cite{Magma} version V2.21-8 for computations, and the
  computations were run on an Intel Xeon W3565 3.2GHz processor. A
  portion of the work in this paper constitued the master's thesis of
  the first author.
\end{ack}

\section{Background}

If $f_{1}(x) = x^{5} + ax + b$ and $f_{2}(x) = x^{5} + cx + d$ are two
trinomials, we say that $f_{1}$ and $f_{2}$ are equivalent
if there is a rational number $\alpha \ne 0$ so that $\frac{1}{\alpha^{5}}
f_{1}(\alpha x) = f_{2}(x)$. If $b$ and $d$ are not zero,
we have that $f_{1}(x)$ is equivalent to $f_{2}(x)$ if and only if
$a^{5}/b^{4} = c^{5}/d^{4}$. This motivates the definition of
the parameter $t = a^{5}/b^{4}$, defined on trinomials 
$f_{1}(x) = x^{5} + ax + b$. If $a \ne 0$ and $b \ne 0$, note that
$f_{1}(x)$ is equivalent to $x^{5} + tx + t$.

If $C$ is a curve of genus $g \geq 2$ with $\Omega$
the sheaf of holomorphic differential 1-forms on $C$, the canonical
ring is
\[
  R(C) = \oplus_{d=0}^{\infty} H^{0}(C,\Omega^{\otimes d}).
\]
We have that $C \cong {\rm Proj}~R(C)$, as the canonical divisor
is ample. We will use the canonical ring of our genus four
curves $C_{K}$ and an automorphism $\tau : C \to C$ of order $2$
to construct the curve quotient $C_{K}/\langle \tau \rangle$.

If $E/K$ is an elliptic curve defined over a number field $K$, the
root number of $E/K$, $w_{E/K}$ is defined in terms of local
$\epsilon$-factors corresponding to representations of local
Weil-Deligne groups. (See \cite{VDok} for more details.) The number
$w_{E/K}$ is conjecturally the sign of the functional equation of the
$L$-function $L(E/K,s)$. Since the Birch and Swinnerton-Dyer
conjecture predicts that the rank of $E(K)$ is equal to the order of
vanishing of $L(E/K,s)$ at $s = 1$, we can conjecturally predict the
parity of the rank of $E/K$ by determining $w_{E/K}$. If $E/K$ is an
elliptic curve and $\rho : \Gal(\overline{\Q}/K) \to \GL_{n}(\C)$ is an
Artin representation, we define $L(E \otimes \rho, s)$ to be the
$L$-function of ${\rm tw}_{\rho}(E)$, the twist of $E$ by
$\rho$ (see \cite{VDok}, Section 3 for details about this definition). 

The technique we will use for provably finding all the rational points
on our genus four curves is elliptic curve Chabauty. This technique
was developed theoretically by Nils Bruin (see \cite{ECChab}) and
implemented in Magma.  The setup for elliptic curve Chabauty is the
following. Given a curve $C/\Q$, it is necessary to have a map
$\Phi : C \to \P^{1}$ defined over $\Q$ that factors as
$\Phi = \pi \circ \varphi$, where $\pi : C \to E$ is a map to an
elliptic curve defined over a number $K$ field, and
$\varphi : E \to \P^{1}$. Elliptic curve Chabauty is sometimes able to
show that $\{ P \in E(K) : \varphi(P) \in \P^{1}(\Q) \}$ is finite and
determine its elements.

\section{The curve $C_{K}$}
\label{crv}

Fix a degree five number field $K/\Q$ with $K =
\Q[\alpha]$.
Trinomials with a root in $K$ are in bijection with elements
$\beta = a + b \alpha + c \alpha^{2} + d \alpha^{3} + e \alpha^{4}$
with the property that the characteristic polynomial of the linear
transformation $T_{\beta} : K \to K$ given by $T_{\beta}(x) = \beta x$
has the form $x^{5} + rx + s$ for $r, s \in \Q$. By representing the
transformation $T_{\beta}$ in terms of the basis
$\{ 1, \alpha, \alpha^{2}, \alpha^{3}, \alpha^{4} \}$, we obtain three
equations involving $a$, $b$, $c$, $d$, and $e$ that express this
condition. For the rest of this section, we will focus on the
important special case that $\alpha$ is a root of $x^{5} + tx + t$.

\begin{proposition}
\label{crveq}
Suppose that $K = \Q[\alpha]$, where $\alpha$ is a root of $x^{5} + tx + t$.
Then the trinomials in $\Q[x]$ with a root in $K$ (up to equivalence)
are in bijection with the rational points on $C_{K}(\Q)$, where $C_{K} \subseteq \P^{3}$ is
defined by
\begin{align*}
  & -5a^{2} + 50ab + 32tbd + 16tc^{2} + 40tcd = 0\\
  & -10a^{3} + 25a^{2}b - 125a^{2}c - 160tacd - 100tad^{2}\\ 
  &+ 64tb^{2} c + 80tb^{2} d + 80tbc^{2} - 64t^{2}cd^{2} - 48t^{2} d^{3} = 0.
\end{align*}
\end{proposition}
\begin{proof}
Applying the procedure above gives the equation $4te = 5a$ from the
assumption that the coefficient of $x^{4}$ in the characteristic polynomial
of $T_{\beta}$ is zero. Plugging this into the equations that result
from the vanishing of the coefficients of $x^{3}$ and $x^{2}$ yield the
equations above. Finally, if $(a_{1} : b_{1} : c_{1} : d_{1}) = \lambda (a_{2} : b_{2} : c_{2} : d_{2})$ are two equivalent points on $C_{K}$, the
corresponding trinomials are related by $f_{1}(x) = \lambda^{5} f_{2}(x/\lambda)$,
and are hence equivalent.
\end{proof}

\begin{remark}
The curve $C_{K}$ is isomorphic over the Galois closure of $K/\Q$ to
the curve $B$ defined by
\begin{align*}
  x_{1} + x_{2} + x_{3} + x_{4} + x_{5} &= 0\\
  x_{1}^{2} + x_{2}^{2} + x_{3}^{2} + x_{4}^{2} + x_{5}^{2} &= 0\\
  x_{1}^{3} + x_{2}^{3} + x_{3}^{3} + x_{4}^{3} + x_{5}^{3} &= 0,\\
\end{align*}
via the map $x_{i} = a + b \alpha_{i} + c \alpha_{i}^{2} + d \alpha_{i}^{3} +
e \alpha_{i}^{4}$, where $\alpha_{1}, \ldots, \alpha_{5}$ are the conjugates
of $\alpha$. This curve is known as Bring's curve after the work
of Erland Samuel Bring (in 1786, as reported by Felix Klein on page 157 of 
\cite{Klein}) on reducing a general quintic to one of the
form $x^{5} + ax + b$. This curve also arose in the work of A. Wiman
(see \cite{Edge}).

The curve $B$ has $120$ automorphisms, and this is the largest number
possible for a curve of genus $4$. Moreover, any genus $4$ curve with
$120$ automorphisms is isomorphic (over $\C$) to $B$ (see
\cite{Breuer}). The Jacobian of $B$ is isogenous to $E_{0}^{4}$, where
$E_{0} : y^{2} = x^{3} - 675x - 79650$ (see Section 8.3.2 of
\cite{Serre}).
\end{remark}

To obtain equations for the map from $C_{K}$ to an elliptic curve $E$,
we first compute an automorphism of $\tau$ of $C_{K}$. This
automorphism is obtained as a composition of
$C_{K} \to B \to B \to C_{K}$, where the map from $B \to B$ is
$(x_{1} : x_{2} : x_{3} : x_{4} : x_{5}) \mapsto
(x_{2} : x_{1} : x_{3} : x_{4} : x_{5})$.
This automorphism of $C_{K}$ is the restriction to $C_{K}$ of an
automorphism of $\P^{3}$ and hence can be represented by a
$4 \times 4$ matrix. This automorphism is defined over the smallest
field $L = \Q[z]$ over which $f(x)$ factors as $g(x) h(x)$, where
$g(x)$ has the roots $\alpha_{1}$ and $\alpha_{2}$. In general, $L/\Q$
has degree $10$, and $z$ is a root of the polynomial
\[
  x^{10} - 3tx^{6} - 11tx^{5} - 4t^{2} x^{2} + 4t^{2}x - t^{2}.
\]
We have obtained the entries of the matrix giving the automorphism
as elements of $L[t]$ - they are too cumbersome to reproduce here.

Next, we will give a method to compute the quotient curve
$C_{K}/\langle \tau \rangle$. This procedure will work on any
non-hyperelliptic genus $4$ curve with an involution $\tau$ so that
the quotient $C_{K}/\langle \tau \rangle$ has genus $1$. Identify
$C_{K}$ with its canonical embedding in $\P^{3}$. By Proposition 2.6
of \cite{Miranda} (page 207), the canonical image of a non-hyperelliptic
genus $4$ curve is the complete intersection of a quadric and
a cubic, $C_{K} : p_{2}(a,b,c,d) = p_{3}(a,b,c,d) = 0$. We consider
the action of $\tau$ on $H^{0}(C,\Omega)$, the $4$-dimensional space
of holomorphic $1$-forms. It is well-known (see Theorem~7.5 on page 27 of
\cite{Breuer}) that the subspace of $H^{0}(C,\Omega)$ fixed by $\tau$
is equal to the genus of the quotient curve. Under the assumption that
$C_{K}/\langle \tau \rangle$ has genus $1$, it follows that $\tau$
(thought of as a linear map $H^{0}(C,\Omega) \to H^{0}(C,\Omega)$) has
a one-dimensional $1$-eigenspace, and hence a $3$-dimensional
$-1$-eigenspace. In general, we let $H^{0}(C,\Omega^{\otimes i})^{+}$
(respectively $H^{0}(C,\Omega^{\otimes i})^{-}$) denote the $+$ and $-$
eigenspaces of $\tau$ acting on $H^{0}(C,\Omega^{\otimes i})$.
Let $\{ e_{1} \}$ be a basis for $H^{0}(C,\Omega)^{+}$ and
$\{ e_{2}, e_{3}, e_{4} \}$ be a basis for $H^{0}(C,\Omega)^{-}$.

\begin{proposition}
Assume the notation above. The morphism $\phi : C \to \P^{2}$ given by
$\phi(a : b : c : d) = (e_{2} : e_{3} : e_{4})$ is the quotient
map. The image of $\phi$ is a cubic curve.
\end{proposition}
\begin{remark}
The authors have computed the coefficients of the plane cubic as elements
of $\Q[z,t]$.
\end{remark}
\begin{proof}
If the coefficient of $e_{1}^{2}$ in $p_{2}$ is zero, then
$e_{1}$ can be solved for in terms of $e_{2}$, $e_{3}$ and $e_{4}$,
which makes the map $\phi : C \to \P^{2}$ have degree $1$. This 
is impossible, however, because $\phi(P) = \phi(\tau(P))$. It follows
that the coefficient of $e_{1}^{2}$ in $p_{2}$ is nonzero and this
implies that $\phi$ has degree $2$. 

We have that $\dim H^{0}(C,\Omega^{\otimes 2})^{+} = 7$ and is spanned
by $e_{1}^{2}$, $e_{2}^{2}$, $e_{3}^{2}$, $e_{4}^{2}$, $e_{2} e_{3}$,
$e_{2} e_{4}$ and $e_{3} e_{4}$. On the other hand $\dim
H^{0}(C,\Omega^{\otimes 2})^{-} = 3$ and is spanned by $e_{1} e_{2}$,
$e_{1} e_{3}$ and $e_{1} e_{4}$. Since the automorphism $\tau$ must
map $p_{2}$ to a $\lambda p_{2}$ for some $\lambda$, and since $p_{2}$
must be irreducible (since $C$ is contained in a unique quadric
surface), it follows that $\tau(p_{2}) = p_{2}$ (because if
$\tau(p_{2}) = -p_{2}$, then $e_{1}$ would be a factor of $p_{2}$). It follows
that $p_{2} = c_{1} e_{1}^{2} + c_{2} e_{2}^{2} + c_{3} e_{3}^{2} + c_{4} e_{4}^{2}
+ c_{5} e_{2} e_{3} + c_{6} e_{2} e_{4} + c_{7} e_{3} e_{4}$. This implies
that if $e_{2}$, $e_{3}$ and $e_{4}$ are all zero, then $e_{1} = 0$,
and so $\phi$ is a morphism. Since this morphism is invariant
under the action of $\tau$, it factors $C \to C/\langle \tau \rangle
\to \im \phi$. The map from $C/\langle \tau \rangle \to \im \phi$
has degree $1$, and must therefore be an isomorphism.

Decompose the polynomial $p_{3} = p_{3}^{+} + p_{3}^{-}$, where
$\tau(p_{3}^{\pm}) = \pm p_{3}^{\pm}$. Note that $p_{3}^{-}$ is
nonzero, since if $p_{3} = p_{3}^{+}$, then $e_{1} | p_{3}^{+}$ and
this implies that $C$ is reducible. Consider the $14$ polynomials
$p_{3}^{-}, e_{2} p_{2}, e_{3} p_{2}, e_{4} p_{2}$ and
$e_{2}^{i} e_{3}^{j} e_{4}^{k}$ with $i+j+k = 3$. All of
these lie in $H^{0}(C,\Omega^{\otimes 3})^{-}$, but $\dim H^{0}(C,\Omega^{\otimes 3})^{-}$ is $13$. It follows that these polynomials are linearly dependent.
This implies that the image of $\phi : C \to \P^{2}$ is a cubic curve,
as desired.
\end{proof}

\section{The elliptic curve $E/L$}

The assumption that $K$ is defined by $x^{5} + tx + t$ forces
$(0 : 1 : 0 : 0)$ to be a rational point on the curve $C$. The image of
this rational point on the plane cubic turns out to be $(0 : 1 : 0)$. 

By putting the plane cubic in Weierstrass form, we find that (for $t \ne -3125/256$),
\[
  E : y^{2} = x^{3} - 675\beta^{2} x - 79650\beta^{3},
\]
where
$\beta = \frac{16}{5}z^{9} + \frac{8}{5}z^{8} + \frac{4}{5} z^{7} -
3z^{6} - \frac{48}{5} tz^{5} - 40tz^{4} - 20tz^{3} +
\frac{13}{5}tz^{2} + (-64/5t^{2} + 29t)z + \frac{32}{5}t^{2}$.
(If $t = -3125/256$, then $x^{5} + tx + t$ has a repeated linear
factor.) This curve is a quadratic twist of
$E_{0} : y^{2} = x^{3} - 675x - 79650$, which has conductor $50$,
$j(E_{0}) = -25/2$, and is one of four elliptic curve over $\Q$ (up to
quadratic twist) with a rational $15$-isogeny. The extension
$L[\sqrt{\beta}]/L$ is a degree $2$ extension inside the splitting
field of $x^{5} + tx + t$. In particular, over the field $L[\sqrt{\beta}]$,
$x^{5} + tx + t$ has two linear factors.

\begin{proposition}
If $t \ne 0$, $t \ne 3125/144$, $E(L)$ has positive rank. 
\end{proposition}
\begin{proof}
If $t = 3125/144$, the point $(0 : 1 : 0)$ is a flex on the plane cubic
found in the previous section. Otherwise, there is a second point $Q$
on the tangent line to the plane cubic at $(0 : 1 : 0)$. Call $P$ the image
of this point on $E$. A straightforward computation shows that 
if $3P = (0 : 1 : 0)$ or $5P = (0 : 1 : 0)$, then $t = 0$ or $3125/144$.
We will show that in all other cases, $P$ has infinite order.

Because $E$ is isomorphic to $E_{0} : y^{2} = x^{3} - 675x - 79650$
over $L[\sqrt{\beta}]$, if $P$ has order $k$, then $E_{0}$ must have a
point of order $k$ defined over a degree $20$ extension. Consider the
mod $\ell$ Galois representation attached to $E_{0}$. If this
representation is surjective, then any field over which $E_{0}$
acquires an $\ell$-torsion point must have degree a multiple of
$\ell^{2} - 1$. This is $> 20$ for any $\ell > 5$. Andrew Sutherland
has verified the surjectivity of the mod $\ell$ Galois representation
attached to $E_{0}$ for $5 < \ell < 80$ (see \cite{Sutherland}).
Moreover, it is known (by work of Serre, see \cite{Serre72}) that if
$\ell > 37$ and the mod $\ell$ Galois representation is not
surjective, there is a quadratic charater $\chi$ unramified at primes
not dividing $50 = N(E_{0})$ for which $a_{p} \equiv 0 \pmod{\ell}$
for any $p \nmid N$ for which $\chi(p) = -1$. It is easy to see (by
considering $p = 3$, $7$ and $11$) that no such $\chi$ exists. Thus,
the mod $\ell$ Galois representation is surjective if $\ell > 5$. It
follows that $P$ is not a torsion point unless $t = 0$ or
$t = 3125/144$.
\end{proof}

The above results raises the hope that for a generic value of $t$,
the rank of $E(L)$ might be one. This is (probably) false for many values of
$t$ as evidenced by the following result.

\begin{proposition}
Suppose that $K/\Q$ is unramified at $2$ and $5$ and the Galois
of $x^{5} + tx + t$ is isomorphic to $F_{20}$, $A_{5}$ or $S_{5}$. 
Then the root number of $w_{E/L} = 1$.
\end{proposition}
\begin{remark}
If $t = -45/4$, then the Galois group of $x^{5} + tx + t$ is isomorphic to
$S_{5}$ and there are at least $5$ trinomials with a root in $K$. The number
field $K/\Q$ is ramified at $2$ and $5$, and the root number $w_{E/L} = -1$,
showing that the hypothesis that $2$ and $5$ are unramified in $K/\Q$ is
necessary.  
\end{remark}

Assuming that the $L$-function of $E/L$ has an analytic continuation and
a functional equation with root number $w_{E/L} = 1$, the Birch and Swinnerton-Dyer conjecture predicts that the rank of $E(L)$ is even. 

\begin{proof}
We will use Corollary~2 of \cite{VDok} which states that if $\rho$ is a
self-dual Artin representation and $E$ is
an elliptic curve defined over $\Q$ with conductor coprime to that of $\rho$,
then
\[
  w_{{\rm tw}_{\rho}}(E) = w_{E}^{\dim \rho} {\rm sign}(\alpha_{p})
  \legen{\alpha_{p}}{N},
\]
where $\Q(\sqrt{\alpha_{p}})$ is the fixed field of $\det \rho$. Suppose
$E/\Q$ is an elliptic curve and $K/\Q$ is a number field with
$K = \Q[\alpha]$, where $\alpha$ has the minimal polynomial $f(x)$.
Then we have $L(E/K,s) = L(E \otimes \rho, s)$, where
$\rho : \Gal(K/\Q) \to S_{n} \to \GL_{n}(\C)$ is the Artin
representation obtained from the permutation character giving the
action of $\Gal(K/\Q)$ on the roots of $f$. We apply this
with $E_{0} : y^{2} = x^{3} - 675x - 79650$. We get that $L(E,s)
= L(E_{0}/L[\sqrt{\beta}],s)/L(E_{0}/L,s)$.

Suppose the Galois group of $x^{5} + tx + t$ over $\Q$ is $S_{5}$. In
this case,
$L(E,s) = L(E_{0} \otimes \rho_{3}, s) L(E_{0} \otimes \rho_{7}, s)$,
where $\rho_{3}$ is the irreducible $4$-dimensional representation
with $\tr \rho_{3}(\sigma) = -2$ when $\sigma$ is a
transposition. Here $\rho_{7}$ is the irreducible $6$-dimensional
representation. Both $\det \rho_{3}$ and $\det \rho_{7}$ are
non-trivial, but there is a unique quadratic subfield of the Galois
closure, and so each of $L(E_{0} \otimes \rho_{3}, s)$ and
$L(E_{0} \otimes \rho_{7}, s)$ has root number
\[
  \sign(\alpha_{p}) \legen{\alpha_{p}}{N}.
\]
Thus, their product has root number $1$, so this is the root number of
$L(E,s)$. 

In the $A_{5}$ case, the representations $\rho_{3}$ and $\rho_{7}$
have trivial determinant (since there are no quadratic subfields) and
are no longer irreducible. However, the same argument applies. In
the case that the Galois group is $F_{20}$, the $L$-function factors as
$L(E_{0} \otimes \rho_{1}, s) L(E_{0} \otimes \overline{\rho}_{1}, s)
L(E_{0} \otimes \rho_{2}, s)^{2}$. Here $\rho_{1}$ and $\overline{\rho}_{1}$
are linear and the product of $L(E_{0} \otimes \rho_{1}, s) L(E_{0} \otimes \overline{\rho}_{1}, s)$ has root number $1$. The representation $\rho_{2}$ is self-dual
and so $L(E_{0} \otimes \rho_{2},s)^{2}$ has root number $1$.
\end{proof}

The final part of Theorem~\ref{main} remains to be established. We have
the map $\phi : C_{K} \to E$. The easiest way to find a map $\rho : C_{K}
\to \P^{1}$ which can be written as
$\rho = \psi \circ \phi$ is to have $\rho$ be the quotient map by
the entire automorphism group of $C_{K}$. As mentioned above
$C_{K}$ is isomorphic to $B$, and it easy to see that the
quotient map $B \to B/{\rm Aut}~B$ is given by
$\xi : B \to \P^{1}$, where $\xi(x_{1} : x_{2} : x_{3} : x_{4} : x_{5})
= (\alpha^{5} : \beta^{4})$, where
\[
  (x - x_{1}) (x - x_{2}) (x - x_{3}) (x - x_{4}) (x - x_{5})
  = x^{5} + \alpha x + \beta.
\]
Composing $\xi$ with the isomorphisms mapping $B$ to and from $C_{K}$
gives that
$\rho : C_{K} \to \P^{1}$ is given by $\phi(a : b : c : d) = (\gamma^{5} : \delta^{4})$, where $x^{5} + \gamma x + \delta$ is the minimal polynomial of
$\beta = a + b \alpha + c \alpha^{2} + d \alpha^{3} + e \alpha^{4}$ (in
the notation of Section~\ref{crv}). This map $\rho$ has degree $120$,
and is the quotient map $C_{K} \to C_{K}/{\rm Aut}~C_{K}$. Therefore,
we can factor it as
$C_{K} \to C_{K}/\langle \tau \rangle \to C_{K}/{\rm Aut}~C_{K} \cong \P^{1}$,
and the map $\psi : C_{K}/\langle \tau \rangle \to \P^{1}$ is the degree
$60$ map we seek. This concludes the proof of Theorem~\ref{main}.

\section{Examples}

In this section, we prove Theorem~\ref{examples}.
\begin{proof}
The polynomial $x^{5} - 5x + 12$ is equivalent to $x^{5} + tx + t$,
where $t = -3125/20736$. The Galois group of $x^{5} - 5x + 12$ is
$D_{10}$, and so $L$, the smallest field over which $f(x)$ factors
as the product of a quadratic and a cubic is $K = \Q[z]$,
where $z^{5} - 5z + 12 = 0$. The elliptic curve $E$ is isomorphic to
$y^{2} = x^{3} - 675x - 79650$ over $K[\sqrt{-10}]$, the splitting
field of the quintic. It follows that $E$ is the $-10$ quadratic twist
of $E_{0}$, and our model for $E$ can be given by
$E : y^{2} = x^{3} - x^{2} - 833x + 109537$. Magma's routines for doing
$2$-descents on curves over number fields determines that
the rank of $E$ is at most $2$, and finds two generators
(one with $x = 10z^{3} + 10z^{2} - 30z + 47$ and the
second with $x = -10z^{3} - 10z^{2} + 30z - 13$). 

The most difficult part of
of the computation is determining equations for the map
$\psi : E \to \P^{1}$. We have in hand equations
for the map $\rho : C_{K} \to \P^{1}$ and equations for the map from
$C_{K} \to E$ and using this we can create a number of pairs
$(P, Q)$ where $P \in E(K)$ and $Q = \psi(P) \in \P^{1}(K)$
by taking a point $P \in E(K)$, computing the two preimages of $P$
on $C_{K}$, and computing the image of one of these points under $\rho$.

Based on a few examples, we observe that
$\psi(-P) = \psi(P)$ and therefore, we expect that
$\psi(x,y)$ has the form $\frac{A(x)}{B(x)}$, where $\deg A, B \leq 30$.
We use linear algebra to compute $A(x)$ and $B(x)$ and we verify
that $\psi(x,y) = \frac{A(x)}{B(x)}$ by checking that their values
agree on $121$ different points in $E(K)$. The ratio of these two
functions is a map from $E \to \P^{1}$ of degree at most $120$,
and since this function takes on the value $1$ at $121$ different points,
we have that $\psi(x,y) = \frac{A(x)}{B(x)}$.

At this point,
we use Nils Bruin's Magma implementation of elliptic curve Chabauty
to find all points $R \in E(K)$ with $\psi(R) \in \P^{1}(\Q)$. There
are five such points in $E(K)$, and the preimages of these points
give eight points in $C_{K}(K)$. The only one of these points
that is rational is $(0 : 1 : 0 : 0)$, and this corresponds to
$x^{5} - 5x + 12$. 

Now, we consider the case where $K = \Q[\sqrt[5]{18}]$. The polynomial
$x^{5} + 750x + 3750$ is equivalent to $x^{5} + (6/5)x + (6/5)$ and so
we take $t = 6/5$ in the calculations in the previous section. The
field $L$ is $K[\sqrt[5]{18}, \sqrt{5}]$. The elliptic curve $E$ is
the quadratic twist of $E_{0}$ by $\frac{-5+\sqrt{5}}{2}$. Computing
the $2$-Selmer group of $E$ over $L$ requires knowing the class group
of a degree $3$ extension of $L$, and we are able to determine this
only by assuming GRH. A 140-hour computation in Magma determines that
the $2$-Selmer rank of $E/L$ is $2$. Most of this time is a systematic search
for $2$-adic points on the curve. There are $5$ rational points
on $C_{K}$, and their images on $E$ generate a subgroup of rank $1$.
Without a second point of infinite order, we cannot apply elliptic curve
Chabauty. Directly point searching on $E$ is not likely to be effective,
and explicitly computing the $2$-covers of $E$ is not feasible computationally.

By taking the curve $C_{K}$ and computing its automorphism group
modulo several primes, we are led to suspect that $C_{K}$ has $8$
automorphisms defined over $L$. This can be confirmed using the method
in \cite{RouseDZB}, Subsection 7.5. Applying these automorphisms to
the points on $C_{K}$ give $16$ points in $C_{K}(L)$ (in two
orbits). The images of these points on $E$ generate a subgroup of rank
$2$. Notably, the points on $C_{K}$ corresponding to trinomials with
$t = 0$ are torsion points on $E$, and
$E(L) \cong \Z \times \Z \times \Z/5\Z$.  To compute the map
$\psi : E \to \P^{1}$ we suppose that it has the form
$\frac{A(x) + yB(x)}{C(x)}$, where $\deg A, C \leq 60$ and
$\deg B \leq 58$. This time, we parametrize the quadric surface containing
$C_{K}$ and using this parametrization create points on $C_{K}$ over
cubic extensions of $L$. To find the coefficients of $A(x)$, $B(x)$
and $C(x)$, we exactly solve a $185 \times 181$ linear system (over
$L$). Solving this system in Magma takes $145$ hours and the list of
the resulting coefficients requires more than 1 megabyte.  As in the
previous case, we verify that $\frac{A(x) + yB(x)}{C(x)}$ is equal to
$\psi$ by verifying that $\frac{A(x) + yB(x)}{C(x)}$ has degree $60$
and that these two maps give the same values on $121$ different
points on $E$.

At this point we can apply elliptic curve Chabauty, which requires
only $127$ seconds to run.  We find that there are nine points in
$E(L)$ that have rational image under $\psi$, and there are $16$
points in $C_{K}(L)$ that map to these. Of these, only $5$ are
rational. These are $(0 : 1 : 0 : 0)$, corresponding to
$x^{5} + (6/5)x + (6/5)$, $(-168/55 : 9/11 : 19/11 : 1)$,
corresponding to $x^{5} - 18$, $(36/35 : -30/7 : 24/7 : 1)$,
corresponding to $x^{5} - 432$, $(-22/15 : -7/6 : -5/4 : 1)$,
corresponding to $x^{5} - 324$, and $(-8/65 : 20/39 : -16/39 : 1)$,
corresponding to $x^{5} - 24$.
\end{proof}

\section{Number fields with more than one trinomial defining them}

Taking the equation of the curve from Proposition~\ref{crveq} and eliminating
the parameter $t$ yields the surface
\begin{align*}
  X : & 20a^{3} c d^{2} + 15a^{3}d^{3} + 128a^{2}b^{2}d^{2} + 128a^{2}bc^{2}d + 
240a^{2}bcd^{2} - 100a^{2}bd^{3} + 32a^{2}c^{4}\\ 
&+ 320a^{2}c^{3}d + 700a^{2}c^{2}d^{2} + 250a^{2}cd^{3} - 128ab^{3}cd - 480ab^{3}d^{2} - 64ab^{2}c^{3}\\ 
&- 720ab^{2}c^{2}d - 600ab^{2}cd^{2} - 500ab^{2}d^{3} - 160abc^{4} - 600abc^{3}d - 1500abc^{2}d^{2}\\ 
&- 2500abcd^{3} + 400ac^{5} + 2000ac^{4}d + 2500ac^{3}d^{2} + 1280b^{4}cd + 1600b^{4}d^{2} + 640b^{3}c^{3}\\ 
&+ 4000b^{3}c^{2}d + 2000b^{3}cd^{2} + 800b^{2}c^{4} + 2000b^{2}c^{3}d = 0.
\end{align*}
If $(a : b : c : d)$ is a point on $X$, and $\alpha$ is a root 
of $x^{5} + tx + t$, where $t = \frac{5a^{2} - 50ab}{32bd + 16c^{2} + 40cd}$,
then
\[
  a + b \alpha + c \alpha^{2} + d \alpha^{3} + e \alpha^{4}
\]
is also a root of a trinomial in $\mathbb{Q}[\alpha]$, where
$e = \frac{5a}{4t}$.

The surface $X$ is (very) singular. It has four lines on it:
\begin{itemize}
\item $a = 10b$, $c = -(3/5)b$. These points correspond to $t = 0$.
\item $a = b = 0$. These also correspond to $t = 0$.
\item $b = (21/32)d$, $c = (-3/4)d$. These correspond to $t = \infty$.
\item $a = (125/16)d$, $c = (-5/4)d$. These correspond to
$t = -3125/256$ (for which $x^{5} + tx + t$ is reducible).
\item $c = d = 0$. These points are singular.
\end{itemize}
The family where $d = tc$ is a fibration of $X$ by curves that (for
all but finitely many $t$) have genus $2$. Viewing $X$ as a
hyperelliptic curve over $\Q(t)$, we find that it has a model in the
form $y^{2} = f(x,t)$, and $y = 0$ is a non-hyperelliptic genus $4$
curve. This implies that a non-singular model $\tilde{X}$ of $X$
is a K3 surface. This implies in addition that $X$ has an automorphism
of order $2$ defined over $\Q$.

In addition, the surface has at least seven rational curves:
\begin{itemize}
\item $a = -\frac{3}{100}s^{4} - \frac{1}{5}s^{3} + s^{2}$,
$b = \frac{3}{100}s^{3} + \frac{1}{5} s^{2} - s$, $c = 0$,
$d = \frac{32}{125}s^{2} + \frac{24}{25}s + \frac{16}{5}$. 
\item $a = \frac{7}{2000}s^{4} + \frac{1}{100}s^{3} + \frac{1}{4}s^{2} + s$,
$b = -\frac{7}{2000} s^{3} - \frac{1}{100}s^{2} - \frac{1}{4}s - 1$,
$d = \frac{8}{625}s^{2} + \frac{2}{125}s - \frac{4}{25}$, $c = -\frac{5}{2}d$.
\item $a = -\frac{1}{250} s^{3} + \frac{2}{25} s^{2} - \frac{1}{2} s + 1$,
$b = -\frac{1}{250} s^{2} + \frac{1}{10}$, $d = -\frac{32}{625}s + \frac{16}{125}$, $c = -\frac{5}{4}d$. This curve is the one that appears in Theorem~\ref{family}.
\item $a = 0$, $b = -\frac{1}{2} s^{2} - \frac{5}{4}s$, $c = s$, $d = 1$. 
\item $a = -5s^{2} - \frac{25}{2}s$,
$b = -\frac{1}{2} s^{2} - \frac{5}{4} s$, $c = s$, $d = 1$. This
curve is contained in the singular locus of $X$.
\item The image on $X$ of the family of trinomials defining
$\Q[\left(r^{3} (r+1)(r-1)^{4}\right)^{1/5}]$ given in \cite{SW2}. For this
curve, $c$ and $d$ have degree $25$, while $a$ and $b$ have degree $24$. 
\item The image of the previous curve under the automorphism. For
this curve, $a$, $b$, $c$ and $d$ all have degree $49$.
\end{itemize}

We conclude by raising a number of questions:
\begin{enumerate}
\item What are equations for $\tilde{X}$?
\item What is the rank of the N\'eron-Severi group of $\tilde{X}$?
\item Above we give five rational curves on $X$. It is conjectured that there
are infinitely many rational curves (defined over $\C$) on any K3 surface. 
Is this true for $\tilde{X}$? Are the $\Q$-rational points on $\tilde{X}$ Zariski dense?
\item Are there infinitely many number fields $K$ that contain roots of
three inequivalent trinomials?
\item Are there infinitely many number fields $K/\Q$ whose Galois closure
has Galois group $D_{10}$ that contain the roots of two trinomials? (We know
of $3$ examples.)   
\end{enumerate}

\bibliographystyle{plain}
\bibliography{triref}

\end{document}